\documentclass{amsart}
\usepackage{amsfonts,amssymb,amsmath,amsthm}
\usepackage{url}
\usepackage{enumerate}

\newtheorem{thrm}{Theorem}[section]
\newtheorem{lem}[thrm]{Lemma}
\newtheorem{prop}[thrm]{Proposition}

\theoremstyle{definition}

\newtheorem{remark}[thrm]{Remark}
\numberwithin{equation}{section}

\author{S. Ruhallah Ahmadi and Bruce Gilligan}
\address{
Department of Mathematics and Statistics, University of Regina,
Regina, Canada S4S 0A2}
\email{seyed.ruhala.ahmadi@gmail.com and gilliganbc@gmail.com}
\thanks{This work was partially supported by an NSERC Discovery Grant for which the authors are grateful.  
We also thank Prof. A. T. Huckleberry and the referee for their comments that led to improvements in the work.}


\begin{document}

\title[Complexifying Lie Group Actions on $G/H$ for $d_{G/H}=2$]{Complexifying Lie group actions on homogeneous  \\  
manifolds of non--compact dimension two}

\begin{abstract}    
If $X$ is a connected complex manifold with $d_X = 2$ that admits a (connected) Lie group $G$ 
acting transitively as a group of holomorphic transformations, then the action extends to an action of the 
complexification $\widehat{G}$ of $G$ on $X$ except when 
either the unit disk in the complex plane 
or a strictly pseudoconcave homogeneous complex manifold is 
the base or fiber of some homogeneous fibration of $X$.       
\end{abstract}  


\maketitle

\section{Introduction}   \label{sect1}

A useful invariant for non--compact manifolds 
in the setting of proper actions of Lie groups    
 is  the   notion of   
{\bf  non--compact dimension}   that  was introduced by Abels 
in \cite{Abe76}; see also \S 2 in \cite{Abe82}.      
We take the dual, but equivalent, approach and do not assume that any Lie group action is necessarily proper.     
For $X$ a connected (real) smooth manifold we define $d_X$ to be the codimension of the 
top non--vanishing homology group of $X$ with coefficients in $\mathbb Z_2$.  
The invariant $d_X$ measures, in a certain sense, how far the manifold $X$ is from being compact  
and this invariant is applied here to study a particular case of the following problem.     

\medskip  
Suppose a Lie group $G$ is acting as a group of holomorphic automorphisms on a complex space $X$.   
One can ask whether the complexification $\widehat{G}$ of $G$ (in the sense of Chapter XVII.5 in \cite{Hoch65}) 
acts holomorphically on a complex space $\widehat{X}$ into which $X$ can be $G$--equivariantly embedded   
as an open subset.   
A special case of this problem is to attempt to choose $\widehat{X} = X$.     
Certainly, the automorphism group of a compact complex manifold is a complex Lie group \cite{BoMo47}  
and in the compact case the action always extends to an action of the complexification of the group involved 
on the compact complex space.          
But such an extension does not exist in all settings, e.g., if $X$ is hyperbolic.   
However, it was proved in Theorem 3.1 in \cite{GH98}     that if $X$ is a 
complex manifold with $d_X$ equal to one that is homogeneous under the action  
of a Lie group $G$ acting by holomorphic transformations, then there is a corresponding 
transitive $\widehat{G}$--action on $X$.            

\medskip  
In this paper we consider complex manifolds that have a transitive action of a real Lie group 
acting as a group of holomorphic transformations and  
that satisfy the condition that $d$ is equal to two.   
We wish to resolve the problem stated above in this setting,   
i.e., determine when one can extend the given action of a  $G$ on 
such a complex manifold $X$ to the holomorphic action of its complexification $\widehat{G}$ on $X$.  
Of course, this is not always possible in the setting of $d=2$, 
as obvious examples are the unit disk in the complex plane, or more generally,  
the complement of a closed ball in higher dimensional projective space which fibers real analytically 
as a disk bundle over a compact manifold.        

\medskip  
It turns out that there is only one other type of example where the extension problem 
has a negative answer arising in the following way.  
Let $B = K/L$ be a compact rank one symmetric space and consider its tangent bundle $T(B)$ together  
with its natural Stein manifold structure and induced $K$--action, see Proposition 2 in \cite{MN63}.      
Let $Y$ be a $K$--equivariant compactification of $T(B)$ to a flag manifold obtained by  
adding an ample divisor.     
Now set $X := Y \setminus B$.  
Heuristically one can view this setting in the following way.   
Since $T(B)$ is Stein, it admits a strictly pseudoconvex exhaustion, which when observed from the other side 
endows $X$ with the structure of a strictly pseudoconcave homogeneous manifold.     
Such manifolds were classified in \cite{HS81}.        
Of course,  the extension problem also fails for homogeneous manifolds where such 
examples occur as fiber or base in any equivariant fibration of the manifolds.     
   
\medskip   
The purpose of this note is to prove the following.    

\begin{thrm}  
Suppose $G$ is a connected Lie group acting transitively and almost effectively  
as a group of holomorphic transformations on a complex manifold 
$X = G/H$ with $d_X = 2$.  
Then the complexification $\widehat{G}$ of $G$ acts holomorphically and transitively on $X$ unless 
the unit disk in the complex plane or a strictly pseudoconcave homogeneous manifold is either 
the fiber or base of some homogeneous fibration of $X$.      
\end{thrm}

\medskip   
The methods that we use are now classical, e.g., see \cite{HO84}, \cite{OR84},  \cite{GH98}, and \cite{Gil95}.  
In particular, we use some basic fibrations from \emph{loc. cit.} along with a Fibration Lemma from \cite{AG94}  
for dealing with the invariant $d$ in the setting of fibrations.               
And the paper is organized as follows.   
In section two we discuss the exceptional flag domains that can occur.  
In section three we present a technical Lemma needed later.  
Finally, the proof of the Theorem is given in section four.

\section{The Exceptional Flag Domains}   \label{sect2}

Suppose $\widehat{G}$ is a connected complex semisimple Lie group, $\widehat{P}$ is a parabolic 
subgroup of $\widehat{G}$ and let  $Y := \widehat{G}/\widehat{P}$ be the corresponding flag manifold.  
Any real form $G$ of $\widehat{G}$ has a finite number of orbits in $Y$ and so at least one orbit 
must be open.   
Such an open orbit is called a flag domain.  
For detailed information about this setting we refer the interested reader to \cite{Wol69} and \cite{FHW}.     
We only recall here that the group ${\rm Aut}(Y)$ can be faithfully represented   
into  ${\rm Aut}(\mathbb P_N)$ for some equivariant embedding of $Y$ into a projective space $\mathbb P_N$   
and so we are dealing with linear groups and are in an algebraic setting throughout.   

\bigskip       
Let $K$ be a maximal compact subgroup of $G$ and $K^{\mathbb C}$ be its complexification in ${\rm Aut}(Y)$.     
Matsuki duality (see \cite{Mat79} and \cite{Mat82}; for a geometric approach see Theorem 8.3.1 in \cite{FHW})      
states that to every $G$--orbit in the flag manifold $Y$  
there is a unique $K^{\mathbb C}$--orbit which intersects the $G$--orbit in a $K$--orbit      
and vice versa.   
If the $G$--orbit is open (resp. closed), then its dual $K^{\mathbb C}$--orbit is closed (resp. open). 
Since there is a unique closed $G$--orbit $A$ in $Y$ which is also a $K$-orbit 
(Theorem 3.5 and Corollary 3.4, \cite{Wol69}),
it follows from Matsuki duality that there is a unique open $K^{\mathbb C}$--orbit $W$ in $Y$ which contains $A$. 
Since the $K^{\mathbb C}$--action is algebraic, the isotropy subgroup for this action is algebraic 
and thus has a finite number of connected components.  
As a consequence $W$ has either one or two ends 
(see \cite{Bor53} when the isotropy group is connected and Proposition 1 \cite{Gil91} 
when the isotropy group has a finite number of connected components).    

\bigskip  
We are interested in the special case of a flag domain $X$ with $d_X =2$. 
Our approach is to study the $K$--orbits in the manifold $X$.   
Now the unique complex $K$--orbit $E$ in the flag domain $X$   is the unique 
orbit of minimal dimension and is a strong deformation retract of $X$, e.g., see    \cite{Wol69}.   
We now use this fact.       

\begin{prop} \label{rk1ss}  
Let $X$ be a flag domain homogeneous under 
a real form $G$ of a complex semisimple Lie group $\widehat{G}$ 
acting as a group of holomorphic transformations and assume $d_X = 2$.  
Then $X$ is a non--compact strictly pseudoconcave homogeneous manifold.   
In particular, $\widehat{G}$ does not act transitively on $X$.   
\end{prop}  

\begin{proof}    
As we noted above, $W$ has either one or two ends and 
we first suppose that $W = K^{\mathbb C}/L^{\mathbb C}$ has two ends, 
where $L^{\mathbb C}$ denotes the isotropy subgroup. 
Since we are in an algebraic setting, 
there exists a parabolic subgroup $\widehat{P}$ of $K^{\mathbb C}$ 
and a character $\chi: \widehat{P} \to \mathbb C^*$ such that $L^{\mathbb C} = {\rm ker}\ \chi$, see \cite{Akh77}.  
In particular, $W$ fibers as a $\mathbb C^*$--bundle over the flag manifold $Q := K^{\mathbb C}/\widehat{P}$.    
Since any $K$--orbit in $W$ is mapped surjectively onto the base $Q$, 
all $K$--orbits in $W$ have real codimension at most two.  
Now suppose there would be a $K$--orbit $B$ in $W$ that has real codimension two.  
We claim that this assumption yields a contradiction.    
Since $Q$ is simply connected, $B$ is a topological section of the $\mathbb C^*$--principal bundle 
$W$ and thus $W$ is topologically trivial.
But $H^1(Q, {\mathcal O}) = 0$, because $Q$ is a flag manifold.      
So the line bundle associated to $W$ by adding a 0--section is analytically trivial.
Thus $W$ as a bundle is also analytically trivial.
This means that all $K$--orbits in $W$ are complex hypersurfaces,  
contradicting the fact that in the flag domain $X$ there is a unique $K$--orbit which is complex, 
see Lemma 5.1 in \cite{Wol69}.    
As a result all $K$--orbits in $W$ have real codimension one, i.e., are real hypersurfaces.  

\bigskip  
We are now in the following setting.  
The group $K^{\mathbb C}$ has one open orbit in $Y$ and two that are closed.  
By Matsuki duality the group $G$ correspondingly has two open orbits and 
one closed orbit with the latter being exactly a real hypersurface $K$--orbit in $W$.     
This setting is described in Theorem 5.2 in \cite{GH09}, where it is proved  
that $Y \cong \mathbb P_n$ and   other than $W$, there are two closed $K^{\mathbb C}$--orbits that are 
a $\mathbb P_{q-1}$ and a $\mathbb P_{p-1}$ with $(q-1)+(p-1)=n-1$. 
Since $X = W \cup \mathbb P_{q-1}$ and $d_X = 2$, it follows that $\mathbb P_{q-1}$ 
is a complex hypersurface in $X$, i.e., $q-1 = n-1$.  
From the above equality we see that $p=1$ and so    
$\mathbb P_{p-1}$ is a point.  
The groups that can occur in this setting are also described in Theorem 5.2 in \cite{GH09}.      

\bigskip  
Now suppose $W$ has one end.   
By the discussion given before the statement of the Proposition, the flag manifold  
$Y$ consists of two $K^{\mathbb C}$--orbits, namely, $W$ and $E$, 
where $W$ is the open orbit and $E$ is a complex hypersurface orbit.     
Since we are in the setting of algebraic group actions on a flag manifold, this 
set up was analyzed in \cite{Akh79} using classification results of Wang \cite{Wan52}.    
The open $K^{\mathbb C}$--orbit is the tangent bundle $T(B)$ of a compact rank one 
symmetric space $B$.     
 \end{proof}   

\begin{remark}    
We refer the reader to  \cite{HS81} for a very detailed description  
of these non--compact strictly pseudoconcave homogeneous manifolds.     
It follows from the classification given there that $X$ is the complement of the closure of a ball in $\mathbb P_n$ which is   
the flag manifold $Y$ containing $X$ or $X = Y \setminus B$ with $Y$    
the compactification of the tangent bundle $T(B)$ of a compact rank one symmetric space $B$.           
The list of these spaces is as follows:    
 the $n$-sphere, $\mathbb S^n$, $n\geq 2$,  
real projective space, $\mathbb {RP}_n$, $n\geq 2$,   
complex projective space $\mathbb P_n$, $n\geq 1$,     
quaternionic projective space $\mathbb {QP}_n$, $n\geq 1$, and     
the Cayley projective plane $F_4 / Spin(9)$.
\end{remark}

\section{A technical Lemma}  \label{sect3}

We will need the following technical result in the proof of the theorem.      

\begin{lem}   \label{f.lemma}   
Let $Y = \widehat{G}/\widehat{H}$ be an orbit of a connected complex semisimple Lie 
group $\widehat{G}$ acting holomorphically on some projective space and assume that a real form $G$ 
of $\widehat{G}$ has an open orbit $X := G/H$ in $Y$ with $d_X = 2$.   
Then  $d_Y \le d_X$.  
\end{lem}  

\begin{proof}  
If $Y$ is compact, then we are done because $d_Y=0$. 
So we assume throughout the rest of the proof that $Y$ is not compact.  
Let $K$ be a maximal compact subgroup of $G$.          
As we are in the setting of algebraic groups, there is a Mostow fibration, see \cite{Mos55}, 
of the homogeneous manifold $X$ as a real vector bundle over a minimal $K$--orbit $M$.    
By assumption $d_X = 2$ and thus $M$ has real codimension two.  
So the generic $K$--orbits in $X$ have real codimension one or two.  

\bigskip 
Suppose these orbits all have codimension two.  
If these orbits are complex, then they are also $K^{\mathbb C}$--orbits, i.e., $Y$ is fibered by flag manifolds.    
One then has a fibration   
\[ 
          Y \; = \; \widehat{G}/\widehat{H} \; \longrightarrow \; \widehat{G}/K^{\mathbb C} \cdot\widehat{H} \; =: \; B ,  
\]  
where $B$ is one dimensional and non--compact.  
If  $B$ is biholomorphic to $\mathbb C$, then $d_Y = d_{\mathbb C} = 2$.    
If  $B = \mathbb C^*$, then $X = Y$ and we are again done.     
In passing, note that $Y$ splits as a product.     
And, if the $K$--orbits are not complex, then by the classification result given in Theorem 5.6 in \cite{GH09}   
it follows that $Y$ is compact and we assumed above that this is not the case.         

\bigskip  
So we now assume that the generic $K$--orbits in $X$ are real hypersurfaces.  
In this setting the complexification $K^{\mathbb C}$ of $K$ has an open orbit $Z$ 
which is Zariski open in its Zariski closure $\overline{Z}$.  
We set $E := \overline{Z} \setminus Z$ and note that $Z \subset Y$, since $K^{\mathbb C}$ is a subgroup 
of $\widehat{G}$ and $Y \not= \overline{Z}$, since $Y$ is not compact.    
We can equivariantly modify $\overline{Z}$ so that the exceptional set $E$ has 
pure codimension one \cite{Ka67} and is equivariantly desingularized \cite{Hir70}.  
By Lemma 2.2 in \cite{HS82} the compact group $K$ acts transitively on 
the components of $E$.  
In particular, the ``infinity'' component $E_{\infty}$ of $E$ is a flag manifold that is 
an orbit of both $\widehat{G}$ and $K^{\mathbb C}$.   
Note that $\widehat{G} \not= K^{\mathbb C}$, since, otherwise, $G = K$ 
would be a compact real form and $X$ would be compact, 
contrary to our assumption that $d_X = 2$.  

\bigskip 
We first consider the case where $\widehat{G}$ is simple and prove the general case later.  
There are only certain possibilities that can arise when two different complex simple groups act 
transitively on a flag manifold and the classification was given by Onishchik \cite{Oni62}.  
As noted for example in the Main Theorem in \cite{Ste82},   
the isotropy of the $\widehat{G}$--action on $E_{\infty}$ is a {\bf maximal} 
parabolic subgroup of $\widehat{G}$.  
As a consequence there is no proper fibration of $E_{\infty}$.  
We now consider what this tells us about $Y$.  

\bigskip  
If $Y$ is Stein, then $\widehat{H}$ is reductive, see \cite{Mat60} and \cite{Oni60}.  
Then $G$ acts transitively on $Y$ \cite{Oni69} and so $X=Y$.     
Otherwise, the Weisfeiler construction \cite{Weis66} yields a fibration $Y = \widehat{G}/\widehat{H} 
\to \widehat{G}/\widehat{I}$ with $\widehat{I}$ a parabolic group containing $\widehat{H}$ 
and $F := \widehat{I}/\widehat{H}$ Stein.     
By virtue of the fact that there is no proper fibration of $E_{\infty}$ it follows that 
the fiber $F$ has complex dimension one.   
Note that $F$ is non--compact, since $Y$ is.   
Suppose first that $F = \mathbb C$ and so $Y$ is a $\mathbb C$--bundle over the flag manifold 
$Q := \widehat{G}/\widehat{I}$.   
Since $Q$ is compact and simply connected, the Fibration Lemma \cite{AG94} implies $d_Y = 2 = d_X$.     
Otherwise, $F = \mathbb C^*$ and $Y$ is a $\mathbb C^*$--bundle over $Q$.  
But then the $G$--action on $F$ has an open orbit and this is only possible if this orbit is all of $F$. 
Thus $X = Y$  and this observation completes the proof when $\widehat{G}$ is simple.    

\bigskip
We now assume that $\widehat{G}$ is semisimple.  
In general, if $G$ is simple, it is not necessarily the case that $\widehat{G}$ is simple.  
However, we claim that this is so in our setting.    
First note that $G$ is simple if and only if $K$ is simple if and only if $K^{\mathbb C}$ is simple.  
If $\widehat{G}$ is not simple, then $E_{\infty}$ is a product on which $K^{\mathbb C}$ acts transitively, e.g.,  
see p. 1147 in \cite{Wol69}.    
It follows that $K^{\mathbb C}$ is not simple {\it loc. cit.} and, as a consequence, we see that if $G$ is simple, 
then $\widehat{G}$ is simple.  
Now let $G = G_1 \times G_2$ be a decomposition of $G$ with $G_1$ simple and 
let $K = K_1 \times K_2$ be the corresponding decomposition of the maximal compact subgroup $K$ 
with $K_i$ a maximal compact subgroup of $G_i$ for $i=1,2$.   
We fiber the open $\widehat{G}$--orbit $Y$ by $\widehat{G}_2$ to get a fibration 
\[  
     p: \;     Y \; = \; \widehat{G}/\widehat{H} \; \stackrel{F}{\longrightarrow} \; \widehat{G}/\widehat{G}_2\cdot\widehat{H}.   
\]   
Since the generic $K$--orbits in $X$ are real hypersurfaces, there are two possibilities.   
The first case occurs when the generic $K_1$--orbit is a real hypersurface in $p(X)$ and $F$ is compact.  
Then the quotient $p(X)$ also has a divisor at infinity that is $\widehat{G}_1$--invariant, 
$G_1$ has an open orbit, and the generic $K_1$--orbits are hypersurfaces.  
In other words we are in the setting we considered above where    
we showed that either $p(X)=p(Y)$ or $d_{p(Y)}=2$.         
Since the fiber $F$ is compact, it follows that either $X=Y$ or $d_Y = 2$.       
Otherwise, the generic $K_2$--orbit in the $G_2$--orbit is a real hypersurface and $p(X)=p(Y)$ is compact.  
Since the $G_2$--orbit in $F$ is open and has a divisor at infinity that is $\widehat{G}_2$--invariant,    
the result now follows by induction on the number of simple factors constituting the group $\widehat{G}$. 
\end{proof}

\begin{remark} 
Suppose $X$ is an open $G$--orbit in $Y := \widehat{G}/\widehat{H}$, where $\widehat{H}$ is a closed 
complex subgroup of the complex semisimple Lie group $\widehat{G}$.  
The question whether $d_Y \le d_X$ holds in this setting without any additional assumptions 
on $d_X$ and $\widehat{H}$ is, as far as we know, open and there are no known counterexamples.     
In passing we note some cases where this question has an affirmative answer.  

\medskip  
If $\widehat{H}$ is reductive algebraic, then $G$ acts transitively on $Y$ \cite{Oni69}.   
The existence of such an open $G$--orbit in $\widehat{G}/\widehat{H}$ implies the existence 
of a proper complex subalgebra $\mathfrak m$ of the Lie algebra $\widehat{\mathfrak g}$ of $\widehat{G}$ 
such that one has a decomposition $\widehat{\mathfrak g} = \mathfrak g + \mathfrak m$ 
which is called a local factorization of $\widehat{G}$.   
These local factorizations have been studied in detail in certain cases; we refer   
the interested reader to  \cite{Oni69}, \cite{Mal77}, \cite{Akh13}, and the references cited there.       

\medskip   
If $\widehat{H}$ is parabolic, then $X$ is a flag domain in the flag manifold $Y$ which is compact.  
In this case one has $0 = d_Y \le d_X$.  
Malyshev \cite{Mal75} proved that if a real form $G$ of inner type (he calls this first category)    
has an open orbit in $\widehat{G}/\widehat{H}$, then $\widehat{H}$ is parabolic.  
An example where $G$ has an open orbit in $\widehat{G}/\widehat{H}$ with $\widehat{H}$   
not parabolic is given in 
Example 5.5 in \cite{Akh13} with $\widehat{\mathfrak h}$ an ideal in $\widehat{\mathfrak p}$, the Lie algebra of a 
parabolic subgroup $\widehat{P}$ that contains $\widehat{H}$.   
The base cycle of $X$ lies in a coset of a maximal compact subgroup of the group $\widehat{P}/\widehat{H}$ 
and a comparison of their dimension yields $d_Y \le d_X$.  
\end{remark}

\begin{remark} 
a) The assumption that {\bf the $G$--orbit is open} is needed.   
If $G$ is a real form in $\widehat{G}$, a complex semisimple Lie group, then $d_{\widehat{G}} > d_G$.  
\bigskip   \newline  
b) {\bf The semisimple assumption} is also needed.   
In the classification of homogeneous complex surfaces \cite{OR84}  the example   
\[     
   Y = \widehat {G}/\widehat{H}\cong \mathbb {C}^2 \supset X = G/H = \mathbb {C}^2\setminus \mathbb {R}^2   
 \]    
 arises in two ways.      
One way is with solvable groups, see Proposition 5.2 in {\it loc. cit.} 
The other way which is relevant to our present considerations is via the real  
form $G = SL_2(\mathbb R) \ltimes \mathbb R^2$ of the mixed group $\widehat{G} = SL_2(\mathbb C) \ltimes \mathbb C^2$, 
see Theorem 6.3 in {\it loc. cit.}   
Here $d_Y = \dim_{\mathbb R} Y = 4 > 3 = d_X$, since $X$ is topologically $\mathbb S^1\times\mathbb R^3$.      
\end{remark}

\section{Proof of the Theorem}   \label{sect4}

With these preparations we can now give the proof of the Theorem.  

\begin{proof}  
Let $G/H \to G/J$ be the $\mathfrak g$--anticanonical fibration, e.g., see \cite{HO84} and \cite{OR84}.  
If $J=G$, then $G$ has the structure of a complex Lie group and is acting holomorphically on $X$ 
and thus $X$ already has the desired form, e.g., see Corollary 4, p. 64 in \cite{HO84}.       
Note that $H$ is discrete in this case.     

\medskip  
So we assume that $G\neq J$.   
Then $G/J$ and $J/H$ are complex manifolds and  it follows from the Fibration Lemma \cite{AG94} 
that $d_{G/J} \le 2$.   
Further, the bundle $G/H^{\circ} \to G/J$ is a locally trivial {\bf holomorphic principal}  fiber bundle with 
structure group the {\bf complex}  Lie group $J/H^{\circ}$, where $H^{\circ}$ denotes the connected 
component of the identity of the group $H$.      
Moreover, the base $G/J$ of the ${\mathfrak g}$--anticanonical fibration is a $G$--orbit in some $\mathbb P_N$  
and is open in the corresponding $\widehat{G}$--orbit $Y:= \widehat{G}/\widehat{J}$,    
where $\widehat{G}$ denotes the smallest connected complex subgroup of the automorphism 
group ${\rm Aut}(\mathbb P_N)$ containing $G$.      

\medskip   
If we show that the $G$--action on the base $G/J$ can be extended to a $\widehat{G}$--action on $G/J$,
then we can extend the $G$--action on $\tilde{X} := G/H^{\circ}$ to a $\widehat{G}$--action   
on $\tilde{X}$ (Proposition 2.6, \cite{GH98}) and this then defines a $\widehat{G}$--action on $X$.   
Note that if $d_{G/J}=0$, then $G/J = \widehat{G}/\widehat{J}$ is compact and thus    
is a flag manifold and we have the desired extension.      
And if $d_{G/J} = 1$, then the result follows by Theorem 3.1 in \cite{GH98}.     
Hence we assume that $d_{G/J} = 2$ throughout the rest of the proof.     

\medskip  
Now, as a consequence of a Theorem of Chevalley \cite{Che51},    
the commutator subgroup $(\widehat{G})'$ of $\widehat{G}$ has closed orbits in $Y$  
and this yields the following fibrations     
\[  
     \begin{array}{ccc}    
     G/J &  \hookrightarrow  & \widehat{G}/\widehat{J} \\  
      {\downarrow }         &  & {\downarrow }  \\  
         G/I      & \hookrightarrow  & \widehat{G}/\widehat{J}\cdot (\widehat{G})' \end{array}  
\]   
where $I := G \cap /\widehat{J}\cdot (\widehat{G})' $.   
The base $\widehat{G}/\widehat{J}\cdot (\widehat{G})' $ is a Stein Abelian Lie group;   
see the Lemma on p.173 in \cite{HO81}.   
Since $G/I$ is open in the connected Abelian Lie group $ \widehat{G}/\widehat{J}\cdot (\widehat{G})'$, 
it follows that $G/I = \widehat{G}/\widehat{J}\cdot (\widehat{G})' $.   
If $\dim G/I  > 0$, then $d_{G/I} > 0$, since $G/I$ is Stein.     
Also $2 = d_{G/J} \ge d_{I/J} + d_{G/I}$ by the Fibration Lemma in \cite{AG94}.  
If $d_{G/I}=2$, then $I/J$ is compact and $X=Y$, so we are done. 
And if $d_{G/I} = 1$, then $d_{I/J} = 1$.  
Then the complexification of $I$ acts transitively 
on $I/J$ by Theorem 3.1 in \cite{GH98} and this implies $\widehat{G}$ acts transitively on $X$.    
Again we are done.  

\medskip  
So we may assume that $I = G$ and  $(\widehat{G)'}$ is acting transitively and as an algebraic group on $Y$.  
Then the orbits of its radical $R_{(\widehat{G})'}$ are closed and these orbits give rise to the following fibrations    
\[  
     \begin{array}{ccc}    
     G/J &  \hookrightarrow  & \widehat{G}/\widehat{J} \\  
      {\downarrow }         &  & {\downarrow }  \\  
         G/L      & \hookrightarrow  & \widehat{G}/\widehat{J}\cdot R_{(\widehat{G})'} \end{array}  
\]
where $L := G \cap \widehat{J}\cdot R_{(\widehat{G})'} $.    
If $\dim L/J > 0$, then we apply Lemma 5 in \cite{Gil95} where it is shown that $d_{{L/J}}=2$   
and either a complex Lie group acts transitively on $G/J$ or else $G/J$ fibers as a unit disk bundle 
over the compact base $G/L$ thus proving the result in this setting.    

\medskip  
Otherwise, $R_{(\widehat{G})'}$ is acting ineffectively and   
we may assume that $G$ and $\widehat{G}$ are both semisimple.
Now let $M$ be the connected subgroup of $G$ corresponding to the maximal complex ideal 
$\mathfrak m := \mathfrak g\cap i\mathfrak g$ in the Lie algebra $\mathfrak g$.   
The group $M$ is algebraic and has closed orbits in $Y$ and the fibers of the fibrations of $X$ 
and $Y$ induced by these $M$ orbits are equal.    
Hence we may assume that $\mathfrak m = (0)$  and   
thus the problem reduces to the setting where $G$ is a real form     
of a connected complex semisimple Lie group $\widehat{G}$.

\medskip  
Now Lemma \ref{f.lemma} applies and $d_Y \le d_X = 2$.   
So $d_Y = 0, 1,$ or 2.  

\medskip     
Suppose first that $d_{Y} = 0$ and so $Y$ is compact.        
Then $X=G/J$ is a flag domain in the flag manifold $Y= \widehat{G}/\widehat{J}$,   
where $G$ is a real form of a semisimple complex Lie group $\widehat{G}$.   
Then $Y$ splits into a product of flag manifolds of the simple factors of $\widehat{G}$ 
(modulo ineffectivity).   
The flag domain $X$ splits into a corresponding product of flag domains in each flag  
manifold, see (p 1147 in \cite{Wol69}).   
Now we can ignore factors $X_i$ that are compact (the complexification of the simple real 
form acts transitively on such factors) and that satisfy $d_{X_i}=1$ (again the complexification 
acts transitively by Theorem 3.1 in \cite{GH98}).  
Thus we are reduced to consideration of one factor having $d=2$ and all other factors compact.   
Such a flag domain of a real form of a complex semisimple Lie group is handled by Proposition  \ref{rk1ss}.

\medskip  
Next assume that $d_Y=1$.    
By the result in \cite{Akh77}  we get the following fibrations    
\[  
\begin{array}{ccccc}    
     X &  \hookrightarrow  & Y \\  
      {\downarrow^{F}}         &  & {\downarrow^{\mathbb C^*}}  \\  
     Z      & \hookrightarrow  & Q
\end{array}  
\]  
where $Q$ is a flag manifold.
Since $F$ is an open orbit in $\mathbb C^*$, we have $F=\mathbb C^*$.  
Hence $d_Z = 1$, and $\widehat{G}$ acts on $Z$ (Theorem 3.1, \cite{GH98}) and 
thus there is a transitive $\widehat{G}$ action on $X$.  
We are again done by our observations above.   

\medskip 
Finally we consider the case $d_Y=2$.   
We are in a setting where a linear algebraic group is acting on $Y$ and    
Akhiezer \cite{Akh83} proved that the complex manifold $Y$ has the following fibration 
which induces a corresponding fibration of $X$:   
\begin{equation*}
   \begin{array}{rcccc}  
    X               & \hookrightarrow  & Y  \\
     {\downarrow F}  &                  & {\downarrow {\widehat F}} \\ 
     G/U             & \hookrightarrow  &   \widehat{G}/\widehat{P}\\  
   \end{array}
\end{equation*}
where $\widehat{P}$ is a parabolic subgroup of $\widehat{G}$,   
$U := \widehat{P}\cap G$ and $\widehat{F}$ is $\mathbb C, \mathbb C^*\times\mathbb C^*$, or 
$\mathbb P_2$ minus a quadric curve.     
Note that since $F$ is open in the Stein space $\widehat{F}$, one has $d_F \neq 0$. 
If $d_F = 1$, then $F = \mathbb C^*$ and this case does not occur.  
Hence the remaining case is when $d_F = 2$.  
If the manifold $\widehat{F}$ is biholomorphic to 
$\mathbb C^*\times \mathbb C^*$, or $\mathbb P^2\setminus Q$ with $Q$ a quadric curve, 
then the transitive action of $\widehat{P}$ on $\widehat{F}$ 
is given by group multiplication in the first case and the action of $PSL(2,\mathbb C)$ by projective 
transformations in the second.    
In both cases there is no Lie subgroup with an open orbit $F$ having $d_F=2$.   
Otherwise, $\widehat{F} = \mathbb C$ with the complex two dimensional affine group acting transitively, see \cite{Akh83},   
and $F$ can then be the upper half plane as the orbit of the real affine group.  
Once again the unit disk occurs as the fiber of a bundle over a flag manifold and this completes the proof of the Theorem.   
\end{proof}

\end{document}